\documentclass[12pt]{article}
 \newtheorem{lem}{Lemma}
    \newtheorem{cor}{Corollary}
    \newtheorem{theo}{Theorem}    
    
    \newtheorem{rem}{Remark}
    \newtheorem{prop}{Proposition}
    \newenvironment{proof}{\noindent{\it Proof.\ }}{\hfill
    {\mbox{$\square$}}\medskip}
 \usepackage{latexsym,amssymb}
 \usepackage{amsmath}
  \newcommand{\g}{\`}

\title{A classification of certain\\almost $\alpha$-Kenmotsu manifolds}
\author{Giulia Dileo}
\date{}

\begin{document}
\maketitle

\textbf{Abstract.} {\small
We study $\mathcal D$-homothetic deformations of almost $\alpha$-Kenmotsu structures. We characterize almost contact metric manifolds which are $CR$-integrable almost $\alpha$-Kenmotsu manifolds, through the existence of a canonical linear connection, invariant under $\mathcal D$-homothetic deformations.
If the canonical connection associated to the structure $(\varphi,\xi,\eta,g)$ has parallel torsion and curvature, then the local geometry is completely determined by the dimension of the manifold and the spectrum of the operator $h'$ defined by $2\alpha h'=({\mathcal L}_\xi\varphi)\circ\varphi$.
In particular, the manifold is locally equivalent to a Lie group endowed with a left invariant almost $\alpha$-Kenmotsu structure. In the case of almost $\alpha$-Kenmotsu $(\kappa,\mu)'$-spaces, this classification gives rise to a scalar invariant depending on the real numbers $\kappa$ and $\alpha$.\\\\
2010 \textit{Mathematics Subject Classification}: 53C15; 53C25.\\
\textit{Key words and phrases}: almost $\alpha$-Kenmotsu manifolds, $\mathcal D$-homothetic deformations, $CR$-manifolds, $\eta$-parallel structures, nullity distributions.}

\section*{Introduction}
Almost Kenmotsu manifolds are a special class of almost contact metric manifolds, recently investigated in \cite{OLS1,KP,DP,DPnull,DPeta}.
An almost contact metric manifold $(M^{2n+1},\varphi,\xi,\eta,g)$ is said to be an almost Kenmotsu manifold if $d\eta=0$ and $d\Phi=2\eta\wedge \Phi$, where $\Phi$ is the fundamental $2$-form associated to the structure.
Normal almost Kenmotsu manifolds are known as Kenmotsu manifolds \cite{Ken}: they set up one of the three classes of almost contact metric manifolds whose automorphism group attains the maximum dimension \cite{Ta}.

The class of almost Kenmotsu manifolds is not invariant with respect to $\mathcal D$-homothetic deformations, that is changes of the structure tensors of the form
\begin{equation}\label{deformation}
\bar\varphi=\varphi,\qquad \bar\xi=\frac{1}{\beta}\xi,\qquad \bar\eta=\beta\eta,\qquad\bar g=\beta g+\beta(\beta-1)\eta\otimes\eta,
\end{equation}
where $\beta$ is a positive constant. These deformations were introduced by Tanno in \cite{Tan1} and largely studied for the class of contact metric manifolds. Indeed, for an almost contact metric structure, such a change preserves the property of being contact metric, K-contact, Sasakian or strongly pseudo-convex $CR$, and the property for the characteristic vector field of a contact metric structure to belong to the $(\kappa,\mu)$-nullity distribution. In \cite{Bo_classification} E. Boeckx provides a full classification of non-Sasakian contact metric $(\kappa,\mu)$-spaces up to $\mathcal D$-homothetic deformations. He associates to each non-Sasakian $(\kappa,\mu)$-space $M$ an invariant $I_M$ depending on the real numbers $\kappa,\mu$, and provides an explicit example of such a space for every dimension $2n+1$ and for every value of the invariant.

In this paper we consider the class of almost $\alpha$-Kenmotsu manifolds \cite{OLS1,KP,JanVan}. They are almost contact metric manifolds with structure $(\varphi,\xi,\eta,g)$ such that $d\eta=0$ and $d\Phi=2\alpha\eta\wedge\Phi$, $\alpha$ being a non-zero real constant.
Applying deformation \eqref{deformation}, one obtains an almost $\frac{\alpha}{\beta}$-Kenmotsu structure.

After some preliminaries on general properties of almost $\alpha$-Kenmotsu manifolds, dealing with the Levi-Civita connection and the Riemannian curvature, also under the hypothesis of local symmetry, we shall focus on some properties which are invariant under $\mathcal D$-homothetic deformations. The first one is the $\eta$-parallelism of the operator $h'=\frac{1}{2\alpha}({\mathcal L}_\xi\varphi)\circ\varphi$, where $\mathcal L$ denotes the Lie derivative. The vanishing of the covariant derivative $\nabla_\xi h'$ is also an invariant property.
If both these conditions are satisfied and $h'\ne0$, then the spectrum of $h'$ is of type $\{0,\lambda_1,-\lambda_1,\dots,\lambda_r,-\lambda_r\}$, each $\lambda_i$ being a positive constant. Denoting by $[0]$ the distribution of the eigenvectors of $h'$ with eigenvalue $0$ and orthogonal to $\xi$, and by $[\lambda_i]$ and $[-\lambda_i]$ the eigendistributions with eigenvalues $\lambda_i$ and $-\lambda_i$ respectively, the  manifold is locally the warped product
$$M'\times_{f_0}M_0\times_{f_1} M_{\lambda_1}\times_{g_1} M_{-\lambda_1}\times_{f_2}\ldots\times_{f_r} M_{\lambda_r}\times_{g_r} M_{-\lambda_r},$$
where $M'$ is an open interval, $M_0$, $M_{\lambda_i}$ and $M_{-\lambda_i}$ are integral submanifolds of the distributions $[0]$, $[\lambda_i]$ and $[-\lambda_i]$.  The warping functions are $f_0=c_0e^{\alpha t}$, $f_i=c_ie^{\alpha(1+\lambda_i)t}$ and $g_i=c'_ie^{\alpha(1-\lambda_i)t}$, with $c_0$, $c_i$ and $c'_i$ positive constants. Moreover, $M_0$ is an almost K\"ahler manifold and the structure is $CR$-integrable if and only if $0$ is a simple eigenvalue or $M_0$ is a K\"ahler manifold (Theorem \ref{theo_warped_product}).

As a special case, we shall consider almost $\alpha$-Kenmotsu manifolds whose characteristic vector field $\xi$ belongs to the $(\kappa,\mu)'$-nullity distribution, that is, for some real numbers $\kappa$, $\mu$, the Riemannian curvature satisfies
\begin{equation}\label{kmu}
R_{XY}\xi=\kappa(\eta(Y)X-\eta(X)Y)+\mu(\eta(Y)h'X-\eta(X)h'Y)
\end{equation}
for all vector fields $X$ and $Y$. Applying a $\mathcal D$-homothetic deformation, condition \eqref{kmu} is preserved up to a change of the real numbers $\kappa,\mu$. We shall see that, for an almost $\alpha$-Kenmotsu $(\kappa,\mu)'$-space, the operator $h'$ is $\eta$-parallel and $\nabla_\xi h'=0$. We also prove that $\kappa\leq-\alpha^2$. If $\kappa=-\alpha^2$, then $h'=0$. If $\kappa<-\alpha^2$ then $\mu=-2\alpha^2$, the structure is $CR$-integrable and the Riemannian curvature is completely determined (Theorem \ref{integrability}).

In order to obtain a local classification of the above manifolds, up to $\mathcal D$-homothetic deformations, we consider in section \ref{section_canonic} an invariant linear connection, called the canonical connection, which was introduced in \cite{DPnull} for almost Kenmotsu manifolds. The existence of this connection characterizes almost contact metric manifolds which are $CR$-integrable almost $\alpha$-Kenmotsu manifolds; it can be viewed as the analogue of the Tanaka-Webster connection in contact geometry. In \cite{BoCho_symm} E. Boeckx and J. T. Cho study Tanaka-Webster parallel spaces, i.e. $CR$-integrable contact metric manifolds for which the Tanaka-Webster connection has parallel torsion and curvature tensors; they prove that these spaces are Sasakian locally $\varphi$-symmetric spaces or non-Sasakian contact metric manifolds such that the characteristic vector field belongs to the $(\kappa,2)$-nullity distribution.

Considering the canonical connection $\widetilde\nabla$ of a $CR$-integrable almost $\alpha$-Kenmotsu manifold, we prove that the torsion $\widetilde T$ is parallel with respect to $\widetilde\nabla$ if and only if the tensor field $h'$ is $\eta$-parallel and satisfies $\nabla_\xi h'=0$.
If, furthermore, the curvature tensor $\widetilde R$ satisfies $\widetilde\nabla\widetilde R=0$, then $\widetilde R$ vanishes and this occurs if and only if
$0$ is a simple eigenvalue of $h'$ or the integral submanifolds of the distribution $[\xi]\oplus[0]$ have constant Riemannian curvature $k=-\alpha^2$ (Theorem \ref{theo_curv}). For a fixed dimension of the manifold, supposing $\widetilde\nabla\widetilde T=0$ and $\widetilde\nabla\widetilde R=0$, we prove that the local geometry is completely determined, up to $\mathcal D$-homothetic deformations, by the spectrum of the operator $h'$ (Theorem \ref{theo_class}). In particular, the manifold is locally equivalent to a solvable non-nilpotent Lie group, which is a subgroup of the affine group $\mbox{\emph{Aff}}(2n+1,\mathbb{R})$,  endowed with a left invariant almost $\alpha$-Kenmotsu structure, whose canonical connection coincides with the left invariant linear connection.

Applying the above classification to almost $\alpha$-Kenmotsu $(\kappa,\mu)'$-spaces, with non-vanishing $h'$, we obtain a scalar invariant $I_M$, depending on the real numbers $\kappa$ and $\alpha$. Together with the dimension of the manifold, $I_M$ determines the local structure up to $\mathcal D$-homothetic deformations. We also show that such a manifold is locally $\mathcal D$-conformal to an almost cosymplectic manifold whose characteristic vector field $\xi$ belongs to the $\kappa_c$-nullity distribution, with $\kappa_c=\kappa+\alpha^2$.

\section{Preliminaries}

An almost contact metric manifold is a differentiable manifold $M^{2n+1}$ endowed with a
structure $(\varphi, \xi, \eta, g)$, given by a tensor field $\varphi$ of type $(1,1)$, a
vector field $\xi$, a $1$-form $\eta$ and a Riemannian metric
$g$ satisfying
\[\varphi^2={}-I+\eta\otimes\xi,\quad \eta(\xi)=1,\quad \varphi(\xi)=0,\quad \eta\circ\varphi=0,\]
\[g(\varphi X,\varphi Y)=g(X,Y)-\eta(X)\eta(Y)\qquad \forall X,Y\in\frak{X}(M).\]
Such a structure is said to be $CR$-integrable if the associated almost $CR$-structure $({\mathcal D},J)$ is integrable, where ${\mathcal D}=Im(\varphi)=Ker(\eta)$ is the $2n$-dimensional distribution orthogonal to $\xi$ and $J$ is the restriction of $\varphi$ to $\cal D$. The structure is normal if the tensor field $N=[\varphi,\varphi]+2d\eta\otimes\xi$ identically vanishes, where $[\varphi,\varphi]$ is
the Nijenhuis torsion of $\varphi$. It is well known that normal almost contact metric manifolds are $CR$-manifolds \cite{Ia}. We refer to \cite{Bl1,Bl2} for more details.

An almost $\alpha$-Kenmotsu manifold is an almost contact metric manifold $M^{2n+1}$ with structure $(\varphi,\xi,\eta,g)$ such that
\begin{equation}\label{def_alpha}
d\eta=0,\qquad d\Phi=2\alpha\eta\wedge \Phi,
\end{equation}
where $\alpha$ is a non-zero real constant and $\Phi$ is the fundamental $2$-form
defined by $\Phi(X,Y)=g(X,\varphi Y)$ for any vector
fields $X$ and $Y$. Normal almost $\alpha$-Kenmotsu manifolds are known as $\alpha$-Kenmotsu manifolds.
Let us consider the $(1,1)$-tensor field
\[h':=\frac{1}{2\alpha}({\mathcal L}_\xi\varphi)\circ\varphi.\]
This operator satisfies $h'(\xi)=0$, it is symmetric and anticommutes with $\varphi$. If $X$ is an eigenvector of $h'$ with eigenvalue $\lambda$, then $\varphi X$ is an eigenvector with eigenvalue $-\lambda$, and thus $\lambda$ and $-\lambda$ have the same multiplicity. If $\lambda\ne 0$, we denote by $[\lambda]$ the distribution of the eigenvectors of $h'$ with eigenvalue $\lambda$; if $\lambda=0$, we denote by $[0]$ the distribution of the eigenvectors of $h'$ with eigenvalue $0$ and orthogonal to $\xi$, which has even rank.

The Levi-Civita connection of $g$ satisfies $\nabla_\xi\varphi=0$, which implies that $\nabla_\xi\xi=0$ and $\nabla_\xi X\in{\mathcal D}$ for any $X\in\mathcal D$. Moreover,
$\nabla_X\xi=\alpha(X+h'X-\eta(X)\xi)$
for any vector field $X$, or equivalently,
\begin{equation}\label{nablaeta}
(\nabla_X\eta)(Y)=\alpha g(X+h'X,Y)-\alpha\eta(X)\eta(Y)
\end{equation}
for all vector fields $X,Y$. From \eqref{def_alpha} it follows that the distribution $\mathcal D$ is integrable with almost K\"ahler leaves. The mean curvature vector field of the integral manifolds of $\mathcal D$ is $H=-\alpha \xi$ and these manifolds are totally umbilical if and only if $h'=0$  \cite{KP}.

An almost $\alpha$-Kenmotsu structure is $CR$-integrable if and only if the tensor $N$ vanishes on $\mathcal D$, or equivalently, the integral manifolds of $\mathcal D$ are K\"ahler manifolds.
In terms of the Levi-Civita connection, the $CR$-integrability of the structure can be characterized by the condition
\begin{equation}\label{nablaphi}
(\nabla_X\varphi)(Y)=\alpha g(\varphi X+\varphi h'X,Y)\xi-\alpha\eta(Y)(\varphi X+\varphi h'X)
\end{equation}
for all vector fields $X,Y$, which is equivalent to the $\eta$-parallelism of the tensor field $\varphi$, that is $g((\nabla_X \varphi)Y,Z)=0$ for any vector fields $X,Y,Z$ orthogonal to $\xi$.

Analogously, the operator $h'$ is said to be $\eta$-parallel if $g((\nabla_Xh')Y,Z)=0$ for every vector fields $X,Y,Z$ orthogonal to $\xi$, and this is equivalent to requiring that
\begin{equation}\label{nablah}
(\nabla_Xh')Y=-\alpha g(Y,h'X+h'^2X)\xi-\alpha\eta(Y)(h'X+h'^2X)+\eta(X)(\nabla_\xi h')Y
\end{equation}
for every vector fields $X,Y$.

Most of the results proved in \cite{DP} for the class of almost Kenmotsu manifolds can be generalized to the class of almost $\alpha$-Kenmotsu manifolds. We omit the proofs since they are similar.
\begin{theo}\label{warped_almost}
Let $(M^{2n+1},\varphi,\xi,\eta,g)$ be an almost $\alpha$-Kenmotsu manifold such that $h'=0$. Then $M^{2n+1}$ is locally a warped product $M'\times_fN^{2n}$, where $N^{2n}$ is an almost K\"ahler manifold, $M'$ is an open interval with coordinate $t$, $f=ce^{\alpha t}$, for some positive constant $c$.
\end{theo}

\begin{prop}\label{AKK}
Let $(M^{2n+1},\varphi, \xi, \eta, g)$  be an almost $\alpha$-Kenmotsu
manifold such that the integral manifolds of $\cal D$ are K\"ahler.
Then, $M^{2n+1}$ is an $\alpha$-Kenmotsu manifold if and only if $h'=0$, or equivalently, $\nabla \xi
=-\alpha\varphi^2$. Therefore, a $3$-dimensional almost $\alpha$-Kenmotsu manifold such that $h'=0$ is an $\alpha$-Kenmotsu manifold.
\end{prop}
Consequently, an $\alpha$-Kenmotsu manifold $M^{2n+1}$ is a warped product of type $M'\times_fN^{2n}$, where $M'$ is an open interval, $N^{2n}$ is a K\"ahler manifold and $f=ce^{\alpha t}$, for some positive constant $c$.

As regards the Riemannian curvature of an almost $\alpha$-Kenmotsu manifold, an easy computation shows that
\begin{equation}\label{RXYxi}
R_{XY}\xi=\alpha^2(\eta(X)(Y+h'Y)-\eta(Y)(X+h'X))+\alpha((\nabla_Xh')Y-(\nabla_Yh')X)
\end{equation}
for every vector fields $X,Y$, which implies that
\[R_{\xi X}\xi=\alpha^2(-\varphi^2X+2h'X+h'^2X)+\alpha(\nabla_\xi h')X.\]
If the almost $\alpha$-Kenmotsu manifold is locally symmetric, then the operator $h'$ satisfies $\nabla_\xi h'=0$, and for any unit eigenvector $X$ of $h'$ with eigenvalue $\lambda$, the $\xi$-sectional curvature is given by
\[K(\xi,X)=-\alpha^2(1+\lambda)^2,\]
which implies that $Ric(\xi,\xi)<0$. The geometry of a locally symmetric almost $\alpha$-Kenmotsu manifold is quite different in the two cases with vanishing or non-vanishing $h'$. Indeed, we have the following results.
\begin{theo}\label{h=0+sym}
Let $(M^{2n+1},\varphi, \xi, \eta, g)$ be a locally symmetric almost
$\alpha$-Kenmotsu manifold. Then, $M^{2n+1}$ is an $\alpha$-Kenmotsu manifold if and only if $h'=0$; in this case the manifold has
constant sectional curvature $k=-\alpha^2$.
\end{theo}
Given an almost $\alpha$-Kenmotsu manifold of constant curvature $k$, it can be proved that $h'=0$, and the above Theorem implies that the structure is normal and $k=-\alpha^2$. In the case of non-vanishing $h'$ we have

\begin{theo}\label{theo_symm}
Let $(M^{2n+1},\varphi,\xi,\eta,g)$ be a locally symmetric almost $\alpha$-Kenmotsu manifold with $h'\ne 0$. Then the operator $h'$ admits the eigenvalues $+1$ and $-1$. If, moreover, the Riemannian curvature satisfies
$R_{XY}\xi=0$ for any $X,Y\in{\mathcal D}$,
then the spectrum of $h'$ is $\{0,1,-1\}$, with $0$ as simple eigenvalue. The distributions $[\xi]\oplus[+1]$ and $[-1]$ are integrable with totally geodesic leaves and $M^{2n+1}$ is locally isometric to the Riemannian product of an $(n+1)$-dimensional manifold of constant curvature $-4\alpha^2$ and a flat $n$-dimensional manifold.
\end{theo}

In the following we consider almost $\alpha$-Kenmotsu manifolds with $\alpha>0$. Notice that if $(\varphi,\xi,\eta,g)$ is an almost $\alpha$-Kenmotsu structure with $\alpha<0$, then $(\varphi,-\xi,-\eta,g)$ is an almost $\alpha'$-Kenmotsu structure with $\alpha'=-\alpha>0$.

\section{$\mathcal D$-homothetic deformations}

Let $(M^{2n+1},\varphi,\xi,\eta,g)$ be an almost $\alpha$-Kenmotsu manifold and $(\bar\varphi,\bar\xi,\bar\eta,\bar g)$ the almost $\frac{\alpha}{\beta}$-Kenmotsu structure obtained by the $\mathcal D$-homothetic deformation \eqref{deformation}. Notice that the operators $h'$ and $\bar h'$ associated to these structures coincide.
Let $\nabla$ and $\bar\nabla$ be the Levi-Civita connections of $g$ and $\bar g$ respectively. We prove that for all vector fields $X,Y$,
\begin{equation}\label{LC}
\bar\nabla_XY=\nabla_XY+\alpha\frac{\beta-1}{\beta}(g(X+h'X,Y)-\eta(X)\eta(Y))\xi.
\end{equation}
Indeed, applying the Koszul formula and $d\eta=0$, we have
\begin{equation*}
\bar\nabla_XY=\nabla_XY+\frac{\beta-1}{\beta}(\nabla_X\eta)(Y)\xi
\end{equation*}
and using \eqref{nablaeta}, we obtain \eqref{LC}. The covariant derivatives of $\varphi$ and $h'$ satisfy
\[(\bar\nabla_X\varphi)(Y)=(\nabla_X\varphi)(Y)+\alpha\frac{\beta-1}{\beta}g(X+h'X,\varphi Y)\xi,\]
\[(\bar\nabla_Xh')(Y)=(\nabla_Xh')(Y)+\alpha\frac{\beta-1}{\beta}g(X+h'X,h'Y)\xi,\]
for all vector fields $X$ and $Y$, so that the property for the tensor fields $\varphi$ and $h'$ to be $\eta$-parallel and the vanishing of the covariant derivative $\nabla_\xi h'$ are invariant under $\mathcal D$-homothetic deformations.

An easy computation shows that the Riemannian curvature tensors $R$ and $\bar R$ of $g$ and $\bar g$ are related by the following formula:
\begin{align}\label{curv_def}
\bar R_{XY}Z &= R_{XY}Z+\alpha\frac{\beta-1}{\beta}g((\nabla_Xh')Y-(\nabla_Yh')X,Z)\xi\\
&\quad +\alpha^2\frac{\beta-1}{\beta}(g(Y+h'Y,Z)-\eta(Y)\eta(Z))(X+h'X)\nonumber\\
&\quad -\alpha^2\frac{\beta-1}{\beta}(g(X+h'X,Z)-\eta(X)\eta(Z))(Y+h'Y)\nonumber
\end{align}
for every vector fields $X,Y,Z$. It follows that $\bar R_{XY}\xi= R_{XY}\xi$ for every vector fields $X,Y$. If $\xi$ belongs to the $(\kappa,\mu)'$-nullity distribution, i.e. the Riemannian curvature tensor satisfies \eqref{kmu}, then $\bar \xi$ belongs to the $(\bar\kappa,\bar\mu)'$-nullity distribution, with
\[\bar\kappa=\frac{\kappa}{\beta^2},\qquad\bar\mu=\frac{\mu}{\beta^2}.\]

Let us analyze now the geometry of almost $\alpha$-Kenmotsu manifolds such that $h'$ is $\eta$-parallel and satisfies $\nabla_\xi h'=0$.
\begin{theo}\label{theo_warped_product}
Let $(M^{2n+1},\varphi,\xi,\eta,g)$ be an almost $\alpha$-Kenmotsu manifold such that $h'$ is $\eta$-parallel and $\nabla_\xi h'=0$. Then the eigenvalues of the operator $h'$ are constant.
Let $\{0,\lambda_1,-\lambda_1,\dots,\lambda_r,-\lambda_r\}$ be the spectrum of $h'$, with $\lambda_i>0$. Then $M^{2n+1}$ is locally the warped product
\begin{equation}\label{warped_product}
M'\times_{f_0}M_0\times_{f_1} M_{\lambda_1}\times_{g_1} M_{-\lambda_1}\times_{f_2}\ldots\times_{f_r} M_{\lambda_r}\times_{g_r} M_{-\lambda_r},
\end{equation}
where $M'$ is an open interval, $M_0$, $M_{\lambda_i}$ and $M_{-\lambda_i}$ are integral submanifolds of the distributions $[0]$, $[\lambda_i]$ and $[-\lambda_i]$ respectively. The warping functions are $f_0=c_0e^{\alpha t}$, $f_i=c_ie^{\alpha(1+\lambda_i)t}$ and $g_i=c'_ie^{\alpha(1-\lambda_i)t}$, with $c_0$, $c_i$ and $c'_i$ positive constants.
Finally, $M_0$ is an almost K\"ahler manifold and the structure is $CR$-integrable if and only if $0$ is a simple eigenvalue or $M_0$ is a K\"ahler manifold.
\end{theo}
\begin{proof}
The result is proved in \cite{DPeta} for almost Kenmotsu manifolds, corresponding to the case $\alpha=1$. Let us consider an almost $\alpha$-Kenmotsu structure $(\varphi,\xi,\eta,g)$, with $\alpha\ne1$, such that $h'$ is $\eta$-parallel and $\nabla_\xi h'=0$. Applying the $\mathcal D$-homothetic deformation \eqref{deformation} with $\beta=\alpha$, we obtain an almost Kenmotsu structure $(\bar\varphi,\bar\xi,\bar\eta,\bar g)$ such that $\bar h'$ is $\bar \eta$-parallel and $\bar\nabla_{\bar\xi} \bar h'=0$, and the result applies to this structure. In particular, the distributions $[0]$, $[\lambda_i]$ and $[-\lambda_i]$ are integrable and for any distinct eigenvalues $\lambda_{i_1},\ldots,\lambda_{i_s}$ of $h'$, the distribution $[\xi]\oplus[\lambda_{i_1}]\oplus\dots\oplus[\lambda_{i_s}]$ is integrable with totally geodesic leaves with respect to $\bar g$; \eqref{LC} implies that such leaves are totally geodesic also with respect to $g$.

Let us consider an eigenvalue $\lambda\ne0$ of $h'$. We prove that the leaves of the distribution $[\lambda]$ are totally umbilical. Indeed, since $[\xi]\oplus[\lambda]$ is totally geodesic, choosing a local orthonormal frame $\{e_i\}$ of $[\lambda]$,  the second fundamental form
satisfies $I\!I(e_i,e_j)= -\alpha(1+\lambda)\delta_{ij}\xi$; the mean curvature vector
field is $H=-\alpha(1+\lambda)\xi$ and, for any $X,Y\in [\lambda]$, we have
$I\!I(X,Y)=g(X,Y)H$, so that the leaves of $[\lambda]$ are totally umbilical.
Since the orthogonal distribution $[\lambda]^\bot$ is integrable with totally geodesic leaves, then $M^{2n+1}$ is locally a warped product $B\times_{f_\lambda}M_\lambda$ such that $TB=[\lambda]^\bot$ and $TM_\lambda=[\lambda]$ (see \cite{Hi}).
We denote by $g_0$ and $\hat{g}$ the
Riemannian metrics on $B$ and $M_\lambda$ respectively, such that the warped
metric is given by $g_0 + f_\lambda^2\hat{g}$. The projection
$\pi:B\times_{f_\lambda} M_\lambda\rightarrow B$ is a Riemannian submersion
with horizontal distribution $\mathcal{H}=[\lambda]^\bot$ and
vertical distribution $\mathcal{V}=[\lambda]$. The mean curvature vector field $H= -\alpha(1+\lambda)\xi$
of the immersed submanifold $(M_\lambda,\hat g)$ is $\pi$-related to
$-\frac{1}{f_\lambda}\,\mathrm{grad}_{g_0}f_\lambda$ (\cite{besse}, 9.104) and thus,
$\alpha(1+\lambda)f_\lambda\xi=\mathrm{grad}_{g_0}f_\lambda$. If $m_\lambda$ is the multiplicity  of $\lambda$, we choose local coordinates
$\{t,x^1,\dots,x^{2n-m_\lambda}\}$ on $B$ such that $\xi=\frac{\partial}{\partial
t}$ and $\frac{\partial}{\partial x^i}\in [\lambda]$ for any
$i=1,\ldots,2n-m_\lambda$. Hence, we get
$f_\lambda=c_\lambda e^{\alpha(1+\lambda)t}$, $c_\lambda>0$.

Now, let us consider $TB=[\xi]\oplus[-\lambda]\oplus\bigoplus_{\mu\ne\pm\lambda}[\mu]$. The distribution $[\xi]\oplus\bigoplus_{\mu\ne\pm\lambda}[\mu]$ is integrable with totally geodesic leaves in $M^{2n+1}$ and $[-\lambda]$ is integrable with totally umbilical leaves in $M^{2n+1}$. Since $B$ is a totally geodesic submanifold of $M^{2n+1}$, these distributions are respectively totally geodesic and totally umbilical in $B$ and, arguing as above, $B$ is locally a warped product. This argument can be applied to each distribution $[\lambda_i]$ and $[-\lambda_i]$, $i\in\{1,\ldots,r\}$, obtaining that $M^{2n+1}$ is locally the warped product
\[N\times_{f_1} M_{\lambda_1}\times_{g_1} M_{-\lambda_1}\times_{f_2}\ldots\times_{f_r} M_{\lambda_r}\times_{g_r} M_{-\lambda_r},\]
where $f_i=c_ie^{\alpha(1+\lambda_i)t}$ and $g_i=c'_ie^{\alpha(1-\lambda_i)t}$, with $c_i$ and $c'_i$ positive constants. The manifold $N$ is a totally geodesic submanifold of $M^{2n+1}$ and it is an integral submanifold of the distribution $[\xi]\oplus[0]$. By Theorem \ref{warped_almost}, $N$ is locally a warped product $M'\times_{f_0}M_0$ of an open interval $M'$ and an almost K\"ahler manifold $M_0$, with $f_0=c_0e^{\alpha t}$, $c_0>0$.
\end{proof}

Under the hypotheses of the above Theorem, applying \eqref{nablah}, we have
\begin{equation}\label{nablahsym_0}
(\nabla_X h')Y-(\nabla_Y h')X=-\alpha \eta(Y)(h'X+h'^2X)+\alpha\eta(X)(h'Y+h'^2Y)
\end{equation}
for any $X,Y\in\frak{X}(M)$.
Now, if we suppose that $\mbox{\emph{Sp}}(h')=\{0,\lambda,-\lambda\}$, with $0$ simple eigenvalue, then  $h'^2=\lambda^2(I-\eta\otimes\xi)$ and thus, from \eqref{nablahsym_0} and \eqref{RXYxi} it follows that
\[R_{XY}\xi=-\alpha^2(1+\lambda^2)(\eta(Y)X-\eta(X)Y)-2\alpha^2(\eta(Y)h'X-\eta(X)h'Y).\]
Hence, we have
\begin{prop}\label{prop_0lambda}
Let $(M^{2n+1},\varphi,\xi,\eta,g)$ be an almost $\alpha$-Kenmotsu manifold such that $h'$ is $\eta$-parallel and $\nabla_\xi h'=0$. If $\mbox{Sp}(h')=\{0,\lambda,-\lambda\}$, with $0$ simple eigenvalue, then $\xi$ belongs to the $(\kappa,\mu)'$-nullity distribution, with $\kappa=-\alpha^2(1+\lambda^2)$ and $\mu=-2\alpha^2$.
\end{prop}

As regards almost $\alpha$-Kenmotsu $(\kappa,\mu)'$-spaces we have the following result.

\begin{theo}\label{integrability}
Let $(M^{2n+1},\varphi,\xi,\eta,g)$ be an almost $\alpha$-Kenmotsu manifold
such that $\xi$ belongs to the $(\kappa,\mu)'$-nullity distribution. Then $\kappa\leq-\alpha^2$.

\medskip
If $\kappa=-\alpha^2$, then $h'=0$ and $M^{2n+1}$ is locally a warped product $M'\times_fN^{2n}$, where $N^{2n}$ is an almost K\"ahler manifold, $M'$ is an open interval with coordinate $t$, $f=ce^{\alpha t}$, for some positive constant $c$.

\medskip
If $\kappa<-\alpha^2$, then $h'\neq 0$, $\mu=-2\alpha^2$ and $Spec(h')=\{0, \lambda,
-\lambda\}$, with $0$ as simple eigenvalue and $\lambda =
\sqrt{-1-\frac{\kappa}{\alpha^2}}$. The operator $h'$ is $\eta$-parallel and satisfies $\nabla_\xi h'=0$. The
integral manifolds of  ${\mathcal D}$ are
K\"ahler manifolds. The distributions  $[\lambda]$ and
$[-\lambda]$ are integrable with totally umbilical leaves; the distributions
$[\xi]\oplus[\lambda]$ and $[\xi]\oplus[-\lambda]$ are integrable
with totally geodesic leaves. Finally, $M^{2n+1}$ is locally isometric to the warped
products
$$B^{n+1}(\kappa+2\alpha^2\lambda)\times_f\mathbb{R}^n,\qquad\mathbb{H}^{n+1}(\kappa-2\alpha^2\lambda)\times_{f'}\mathbb{R}^n,$$
where $B^{n+1}(\kappa+2\alpha^2\lambda)$ is a
space of constant curvature $\kappa+2\alpha^2\lambda\leq0$, tangent to the distribution $[\xi]\oplus[-\lambda]$, $\mathbb{H}^{n+1}(\kappa-2\alpha^2\lambda)$
is the hyperbolic space of constant curvature $\kappa-2\alpha^2\lambda<-\alpha^2$, tangent to the distribution $[\xi]\oplus [\lambda]$,
$f=ce^{\alpha(1+\lambda)t}$ and $f'=c'e^{\alpha(1-\lambda)t}$, with $c,c'$
positive constants.
\end{theo}

\begin{proof}
The result is proved in \cite{DPnull} for almost Kenmotsu manifolds.
Let us consider an almost $\alpha$-Kenmotsu structure $(\varphi,\xi,\eta,g)$, with $\alpha\ne 1$, such that $\xi$ belongs to the $(\kappa,\mu)'$-nullity distribution. Applying the $\mathcal D$-homothetic deformation \eqref{deformation} with $\beta=\alpha$, we obtain an almost Kenmotsu structure $(\bar\varphi,\bar\xi,\bar\eta,\bar g)$ such that $\bar\xi$ belongs to the $(\bar\kappa,\bar\mu)'$-nullity distribution, with $\bar\kappa=\frac{\kappa}{\alpha^2}$, and $\bar\mu=\frac{\mu}{\alpha^2}$. Then $\bar\kappa\leq -1$.
If $\bar\kappa= -1$, or equivalently $\kappa=-\alpha^2$, then $h'=0$ and we apply Theorem \ref{warped_almost}.

If $\bar\kappa<-1$ then $\bar h'\neq 0$, $\bar\mu=-2$ and $Spec(\bar h')=\{0, \lambda,
-\lambda\}$, with $0$ as simple eigenvalue and $\lambda =
\sqrt{-1-\bar\kappa}$. The tensor fields $\varphi$ and $h'$ are $\eta$-parallel and $\nabla_\xi h'=0$, since these properties are invariant under $\mathcal D$-homothetic deformations; in particular, the integral manifolds of $\mathcal D$ are K\"ahler manifolds. From Theorem \ref{theo_warped_product} it follows that $M^{2n+1}$ is locally the warped product
\[M'\times_f M_\lambda\times_{f'} M_{-\lambda},\]
where $M'$ is an open interval, $M_\lambda$ and $M_{-\lambda}$ are integral submanifolds of the distributions $[\lambda]$ and $[-\lambda]$ respectively, $f=ce^{\alpha(1+\lambda)t}$ and $f'=c'e^{\alpha(1-\lambda)t}$, with $c,c'$
positive constants.

We compute now the Riemannian curvature of $M^{2n+1}$. Recall that the integral submanifolds of the distribution $[\xi]\oplus[\lambda]$ have constant Riemannian curvature $\bar \kappa-2\lambda$ with respect to the deformed Riemannian metric $\bar g$. Let us compute the relation between the curvature tensors $R$ and $\bar R$ of $g$ and $\bar g$ respectively. Combining \eqref{RXYxi} with the $(\kappa,\mu)'$-nullity condition, $\mu=-2\alpha^2$,  we get
\begin{equation*}
\alpha((\nabla_Xh')Y-(\nabla_Yh')X)=(\kappa+\alpha^2)(\eta(Y)X-\eta(X)Y)-\alpha^2(\eta(Y)h'X-\eta(X)h'Y),
\end{equation*}
and thus, applying \eqref{curv_def}, we obtain
\begin{align*}
\bar R_{XY}Z &= R_{XY}Z+\alpha(\alpha-1)(\eta(Y)g(X-h'X,Z)-\eta(X)g(Y-h'Y,Z))\xi\\
&\quad +\kappa\frac{\alpha-1}{\alpha}(\eta(Y)g(X,Z)-\eta(X)g(Y,Z))\xi\\
&\quad +\alpha(\alpha-1)(g(Y+h'Y,Z)-\eta(Y)\eta(Z))(X+h'X)\\
&\quad -\alpha(\alpha-1)(g(X+h'X,Z)-\eta(X)\eta(Z))(Y+h'Y)
\end{align*}
for any $X,Y,Z\in\frak{X}(M)$.  On the distribution $[\xi]\oplus[\lambda]$ we have $h'=\lambda(I-\eta\otimes \xi)$ and applying the above formula, for any $X,Y,Z\in[\xi]\oplus[\lambda]$, we get \[R_{XY}Z=-\alpha^2(1+\lambda)^2(g(Y,Z)X-g(X,Z)Y).\]
Therefore, the leaves of the distribution $[\xi]\oplus[\lambda]$ have constant Riemannian curvature $-\alpha^2(1+\lambda)^2=\kappa-2\alpha^2\lambda<-\alpha^2$  with respect to $g$ and analogously, the leaves of the distribution $[\xi]\oplus[-\lambda]$ have constant Riemannian curvature $-\alpha^2(1-\lambda)^2=\kappa+2\alpha^2\lambda\leq0$.
Then, $M^{2n+1}$ is locally isometric to the warped products
$$\mathbb{H}^{n+1}(\kappa-2\alpha^2\lambda)\times_{f'}M_{-\lambda},\qquad
B^{n+1}(\kappa+2\alpha^2\lambda)\times_fM_\lambda.$$
We prove that the fibers of the two warped products are flat Riemannian spaces. Denote by $g_0$ and $\hat{g}$ the
Riemannian metrics on $\mathbb{H}^{n+1}(\kappa-2\alpha^2\lambda)$ and $M_\lambda$ respectively, such that the first warped
metric is given by $g_0 + {f'}^2\hat{g}$. Applying Proposition 7.42 in \cite{ONeill}, for any
$U,V,W\in[-\lambda]$, we have
\[\hat R_{UV}W=R_{UV}W-\frac{\|\mathrm{grad }
f'\|^2}{{f'}^2}(g(U,W)V-g(V,W)U).\] On the other hand, $R_{UV}W=-\alpha^2(1-\lambda)^2(g(V,W)U-g(U,W)V)$ and $\|\mathrm{grad }
f'\|^2=\alpha^2(1-\lambda)^2{f'}^2$. Then,
$\hat R_{UV}W=0$. Analogously, the fibers of the second warped product are flat Riemannian spaces.
\end{proof}

Under the hypotheses of the above Theorem, if $\lambda=1$ then both the distributions $[\xi]\oplus[+1]$ and $[-1]$ are
integrable with totally geodesic leaves and the manifold turns out to be locally isometric to the Riemannian product
$\mathbb{H}^{n+1}(-4\alpha^2)\times \mathbb{R}^n$, which is locally symmetric. Conversely, supposing that $M^{2n+1}$ is locally
symmetric, then, by Theorem \ref{theo_symm}, $\lambda=1$ and
$M^{2n+1}$ is locally isometric to $\mathbb{H}^{n+1}(-4\alpha^2)\times
\mathbb{R}^n$. Hence, we have

\begin{cor}\label{cor_symm}
Let $(M^{2n+1},\varphi,\xi,\eta,g)$ be an almost $\alpha$-Kenmotsu manifold
such that $h'\ne0$ and $\xi$ belongs to the $(\kappa,\mu)'$-nullity distribution, $\mu=-2\alpha^2$.
Then $M^{2n+1}$ is locally symmetric if and only if
$Spec(h')=\{0,1,-1\}$, or equivalently $\kappa=-2\alpha^2$, in which case the manifold is locally isometric to
${\mathbb H}^{n+1}(-4\alpha^2) \times {\mathbb R}^n$.
\end{cor}

As another consequence of Theorem \ref{integrability}, we can obtain more information on the Riemannian curvature of an almost $\alpha$-Kenmotsu manifold $(M^{2n+1},\varphi,\xi,\eta,g)$ such that $h'$ is $\eta$-parallel and $\nabla_\xi h'=0$, as in the hypotheses of Theorem \ref{theo_warped_product}. Indeed, for any eigenvalue $\lambda$ of the operator $h'$, the distribution $[\xi]\oplus[\lambda]\oplus[-\lambda]$ is integrable with totally geodesic leaves which inherit an almost $\alpha$-Kenmotsu structure from $M^{2n+1}$.
If $\lambda=0$, then the distribution $[\xi]\oplus[\lambda]\oplus[-\lambda]$ reduces to $[\xi]\oplus[0]$ and the leaves are local warped products $M'\times_{f_0} M_0$, where $M_0$ is a K\"ahler manifold in hypothesis of $CR$-integrability.
If $\lambda>0$ then, by Proposition \ref{prop_0lambda}, the leaves of $[\xi]\oplus[\lambda]\oplus[-\lambda]$ are almost $\alpha$-Kenmotsu manifolds with characteristic vector field belonging to the $(\kappa,\mu)'$-nullity distribution, with $\kappa=-\alpha^2(1+\lambda^2)$ and $\mu=-2\alpha^2$.
By Theorem \ref{integrability}, the leaves of $[\xi]\oplus[\lambda]$ have constant Riemannian curvature $\kappa-2\alpha^2\lambda$ and the leaves of $[\xi]\oplus[-\lambda]$ have constant Riemannian curvature $\kappa+2\alpha^2\lambda$.

\section{The canonical connection}\label{section_canonic}

\begin{theo}\label{tildenabla}
Let $(M^{2n+1},\varphi,\xi,\eta,g)$ be an almost contact metric
manifold. Then $M^{2n+1}$ is a $CR$-integrable almost $\alpha$-Kenmotsu
manifold if and only if there exists a linear connection $\widetilde
\nabla$ such that the tensor fields $\varphi$, $g$, $\eta$ are parallel with respect to $\widetilde\nabla$
and the torsion $\widetilde T$ satisfies:
\begin{itemize}
\item[{\rm a)}] $\widetilde T (X,Y)=0$, for any $X,Y \in \mathcal D$,\vspace{-0.3cm}
\item[{\rm b)}] $2\widetilde T (\xi,X)= \alpha(X + h'X)$, for any $X \in \mathcal D$,\vspace{-0.3cm}
\item[{\rm c)}] $\widetilde T _{\xi}$ is selfadjoint.
\end{itemize}
The connection $\widetilde\nabla$ is invariant under $\mathcal D$-homothetic deformations and it is uniquely determined by
\begin{equation}\label{canonic}
\widetilde \nabla_X Y = \nabla_X Y + \alpha g(X+h'X,Y)\xi -\alpha \eta(Y)(X+h'X),
\end{equation}
where $\nabla$ is the Levi-Civita connection. The connection $\widetilde\nabla$ will be called the canonical connection associated to the structure $(\varphi,\xi,\eta,g)$.
\end{theo}

\begin{proof}
The result of existence and uniqueness of the connection is proved in \cite{DPnull} for almost Kenmotsu manifolds. Let $(\bar\varphi,\bar\xi,\bar\eta,\bar g)$ be the almost contact metric structure obtained from $(\varphi,\xi,\eta,g)$ through deformation \eqref{deformation} with $\beta=\alpha$. Then $(\varphi,\xi,\eta,g)$ is a $CR$-integrable almost $\alpha$-Kenmotsu structure if and only if $(\bar\varphi,\bar\xi,\bar\eta,\bar g)$ is a $CR$-integrable almost Kenmotsu structure, and this is equivalent to the existence of a unique linear connection $\widetilde\nabla$ such that the tensor fields $\bar\varphi$, $\bar g$ and $\bar\eta$ are parallel with respect to $\widetilde\nabla$, and the torsion $\widetilde T$ vanishes on $\mathcal D$ and satisfies
\begin{itemize}
\item[$\mathrm b'$)] $2\widetilde T (\bar\xi,X)= X +\bar h'X$, for any $X \in \mathcal D$,\vspace{-0.2cm}
\item[$\mathrm c'$)] $\widetilde T _{\bar\xi}$ is selfadjoint with respect to $\bar g$.
\end{itemize}
The parallelism of $\bar\varphi$, $\bar g$, $\bar\eta$ is equivalent to the parallelism of $\varphi$, $g$, $\eta$ and $\mathrm b'$) is obviously equivalent to b). Moreover, for any vector fields $X,Y$, we have
\[\bar g(\widetilde T_{\bar \xi}X,Y)=g(\widetilde T_\xi X,Y)+(\alpha-1)\bar\eta(\widetilde T_{\bar\xi} X)\eta(Y).\]
If $\widetilde T _{\bar\xi}$ is selfadjoint with respect to $\bar g$, then $\widetilde T _\xi$ is selfadjoint with respect to $g$ since $\bar\eta(\widetilde T_{\bar\xi} X)=\bar g(X,\widetilde T_{\bar\xi} \bar\xi)=0$. Hence, $\mathrm c'$) implies c). Analogously, one verifies that c) implies $\mathrm c'$).

Denoting by $\bar\nabla$ the Levi-Civita connection of $\bar g$, for any vector fields $X$ and $Y$, we have
$$\widetilde \nabla_X Y = \bar\nabla_X Y + \bar g(X+\bar h'X,Y)\bar\xi - \bar\eta(Y)(X+\bar h'X),$$
and applying \eqref{LC} with $\beta=\alpha$, we get \eqref{canonic}. Finally, the connection is invariant under $\mathcal D$-homothetic deformations. Indeed, if $\widetilde\nabla$ is the canonical connection associated to the almost $\alpha$-Kenmotsu structure $(\varphi,\xi,\eta,g)$, it can be easily verified that $\widetilde\nabla$ satisfies the axioms defining the canonical connection associated to the almost $\frac{\alpha}{\beta}$-Kenmotsu structure $(\bar\varphi,\bar\xi,\bar\eta,\bar g)$ obtained through a $\mathcal D$-homothetic deformation of constant $\beta$.
\end{proof}

Now, let $(M^{2n+1},\varphi,\xi,\eta,g)$ be a $CR$-integrable almost $\alpha$-Kenmotsu manifold. Let $\widetilde\nabla$ be the canonical connection and $\widetilde R$ its curvature tensor. A straightforward computation using \eqref{canonic} shows that for every vector fields $X,Y,Z$
\begin{align}\label{Rcanonic}
\widetilde R_{XY}Z&= R_{XY}Z+\alpha^2(g(Y+h'Y,Z)(X+h'X)-g(X+h'X,Z)(Y+h'Y))\\
&\quad+\alpha g((\nabla_Xh')Y-(\nabla_Yh')X,Z)\xi-\alpha \eta(Z)((\nabla_Xh')Y-(\nabla_Yh')X),\nonumber
\end{align}
where $R$ is the Riemannian curvature tensor. Consequently, we obtain the following result.

\begin{prop}\label{Rtilde0}
Let $(M^{2n+1},\varphi,\xi,\eta,g)$ be an $\alpha$-Kenmotsu manifold. Then the following conditions are equivalent:
\begin{itemize}
\item[\emph{a)}] $\widetilde\nabla\widetilde R=0$,
\item[\emph{b)}] $\widetilde R=0$,
\item[\emph{c)}] $M^{2n+1}$ has constant Riemannian curvature $k=-\alpha^2$,
\item[\emph{d)}] $M^{2n+1}$ is a locally symmetric Riemannian manifold.
\end{itemize}
\end{prop}
\begin{proof}
Since the structure is normal, the operator $h'$ vanishes and the equivalence of b) and c) immediately follows from \eqref{Rcanonic}.  In order to prove the equivalence of a) and b), we show that for any $X\in\mathcal D$, $\widetilde \nabla_\xi X=\alpha X$. Indeed, the manifold is locally a warped product of an open interval $M'$, which is tangent to the vector field $\xi$, and a K\"ahler manifold $N^{2n}$, orthogonal to $\xi$. Therefore, $[\xi,X]=0$ for any $X\in\mathcal D$ and applying b) of Theorem \ref{tildenabla}, we have $\widetilde\nabla_\xi X=2\widetilde T(\xi,X)=\alpha X$. We also notice that, since $\widetilde\nabla\varphi=0$, then $\widetilde\nabla_ZX\in{\mathcal D}$ for any $X\in{\mathcal D}$ and $Z\in\frak{X}(M)$. Hence, for any $X,Y,Z\in\mathcal D$, $\widetilde R_{XY}Z\in\mathcal D$. Supposing $\widetilde\nabla\widetilde R=0$,  from $(\widetilde \nabla_\xi\widetilde R)(X,Y,Z)=0$ we get $\widetilde R_{XY}Z=0$. On the other hand, $\widetilde R_{XY}\xi=\widetilde R_{\xi X}Y=0$ for any vector fields $X,Y$, and thus the curvature tensor $\widetilde R$ vanishes. The equivalence of c) and d) is a consequence of Theorem \ref{h=0+sym}.
\end{proof}

%

We shall discuss now the geometry of $CR$-integrable almost $\alpha$-Kenmotsu manifolds such that $h'\ne0$ and $\widetilde\nabla\widetilde T=0$. First of all we prove the following Lemma.

\begin{lem}\label{lem_canonic}
Let $(M^{2n+1},\varphi,\xi,\eta,g)$ be a $CR$-integrable almost $\alpha$-Kenmotsu manifold. Then the following conditions are equivalent:
\begin{itemize}
\item[\emph{a)}] $\widetilde\nabla\widetilde T=0$,
\item[\emph{b)}] $\widetilde\nabla h'=0$,
\item[\emph{c)}] the tensor field $h'$ is $\eta$-parallel and $\nabla_\xi h'=0$.
\end{itemize}
\end{lem}
\begin{proof}
Recall that $\widetilde\nabla_ZX\in{\mathcal D}$ for any $X\in{\mathcal D}$ and $Z\in\frak{X}(M)$. On the other hand, the torsion $\widetilde T$ vanishes on $\mathcal D$ and thus $(\widetilde\nabla_Z\widetilde T)(X,Y)=0$ for any $X,Y\in{\mathcal D}$ and $Z\in\frak{X}(M)$. Now, applying $\widetilde \nabla\xi=0$ and b) of Theorem \ref{tildenabla}, for any $X\in\mathcal D$ and $Z\in\frak{X}(M)$, we have
\[2(\widetilde\nabla_Z\widetilde T)(\xi,X)=\alpha\widetilde\nabla_Z(X+h'X)-\alpha(\widetilde\nabla_ZX+h'(\widetilde\nabla_ZX))=\alpha(\widetilde\nabla_Zh')X,\]
which proves the equivalence of a) and b), since $(\widetilde\nabla_Zh')\xi=0$. Applying \eqref{canonic}, we get
\[(\widetilde\nabla_Zh')X=\alpha\{(\nabla_Zh')X+\alpha g(h'Z+h'^2Z,X)\xi\}\in\mathcal D\]
and thus, the covariant derivative $\widetilde\nabla h'$ vanishes if and only if for any $X,Y\in\mathcal D$ and $Z\in\frak{X}(M)$, $g((\nabla_Zh')X,Y)=0$, which is equivalent to requiring the $\eta$-parallelism of the tensor field $h'$ and the vanishing of the covariant derivative $\nabla_\xi h'$.
\end{proof}

Under the hypotheses of the above Lemma, Theorem \ref{theo_warped_product} implies that $M^{2n+1}$ is locally isometric to the warped product \eqref{warped_product}, where $M_0$ has dimension $0$ or it is a K\"ahler manifold. For any eigenvalue $\lambda$ of the operator $h'$, each integral submanifold $N$ of the distribution $[\xi]\oplus[\lambda]\oplus[-\lambda]$ is auto-parallel with respect to the canonical connection $\widetilde\nabla$; moreover, the connection induced by $\widetilde\nabla$ on $N$ coincides with the canonical connection associated to the induced almost $\alpha$-Kenmotsu structure.

Let us investigate now the properties of the curvature tensor $\widetilde R$ in the non-normal case.

\begin{theo}\label{theo_curv}
Let $(M^{2n+1},\varphi,\xi,\eta,g)$ be a $CR$-integrable almost $\alpha$-Kenmotsu manifold such that $h'\ne0$ and the canonical connection $\widetilde\nabla$ has parallel torsion. Then the following conditions are equivalent:
\begin{itemize}
\item[\emph{a)}] $\widetilde\nabla\widetilde R=0$,
\item[\emph{b)}] $\widetilde R=0$,
\item[\emph{c)}]$0$ is a simple eigenvalue of $h'$ or the integral submanifolds of the distribution $[\xi]\oplus[0]$ have constant Riemannian curvature $k=-\alpha^2$.
\end{itemize}
\end{theo}
\begin{proof}
By Lemma \ref{lem_canonic}, the operator $h'$ is $\eta$-parallel and $\nabla_\xi h'=0$. Hence, applying \eqref{nablahsym_0} and \eqref{Rcanonic}, we obtain
that the curvature tensors $R$ and $\tilde R$ are related by
\begin{align}\label{Rcanonic_parallel}
\widetilde R_{XY}Z&= R_{XY}Z-\alpha^2 (\eta(Y)g(h'X+h'^2X,Z)-\eta(X)g(h'Y+h'^2Y,Z))\xi\\
&\quad +\alpha^2\eta(Z)(\eta(Y)(h'X+h'^2X)-\eta(X)(h'Y+h'^2Y))\nonumber\\
&\quad+\alpha^2(g(Y+h'Y,Z)(X+h'X)-g(X+h'X,Z)(Y+h'Y))\nonumber
\end{align}
for any $X,Y,Z\in\frak{X}(M)$. We know that $M^{2n+1}$ is locally isometric to the warped product \eqref{warped_product}, where $M_0$ has dimension $0$ or it is a K\"ahler manifold.
Let us consider an eigenvalue $\lambda\ne0$ of $h'$ and the warped product $B\times_fM_\lambda$ such that $TB=[\lambda]^\bot$, $TM_\lambda=[\lambda]$ and $f=c e^{\alpha(1+\lambda)t}$, $c>0$. From  Proposition 7.42 in \cite{ONeill} it follows that for any $X,Y\in TB$ and $V,W\in[\lambda]$,
\[R_{XY}V=R_{VW}X=0,\quad R_{VX}Y=-\frac{H^f(X,Y)}{f}V,\quad R_{XV}W=-\frac{g(V,W)}{f}\nabla_X(\mathrm{grad} f).\]
For any vector field $Z$, we have $Z(f)=\eta(Z)\xi(f)=\alpha(1+\lambda)\eta(Z)f$. Therefore,
\begin{align*}
H^f(X,Y)&= X(Yf)-(\nabla_XY)(f)\\
&= \alpha(1+\lambda)(X(\eta(Y))f+\alpha(1+\lambda)\eta(X)\eta(Y)f-\eta(\nabla_XY)f)\\
&= \alpha(1+\lambda)((\nabla_X\eta)(Y)+\alpha(1+\lambda)\eta(X)\eta(Y))f\\
&= \alpha^2(1+\lambda)(g(X+h'X,Y)+\lambda\eta(X)\eta(Y))f,
\end{align*}
where we used \eqref{nablaeta}. Hence,
\[R_{VX}Y=-\alpha^2(1+\lambda)(g(X+h'X,Y)+\lambda\eta(X)\eta(Y))V.\]
Since $\mathrm{grad} f=\alpha(1+\lambda)f\xi$, then
\begin{align*}
\nabla_X(\mathrm{grad} f)&=\alpha(1+\lambda)\{\alpha(1+\lambda)\eta(X)f\xi+\alpha f(X+h'X-\eta(X)\xi)\}\\
&= \alpha^2(1+\lambda)(X+h'X+\lambda\eta(X)\xi)f,
\end{align*}
and thus
\[R_{XV}W=-\alpha^2(1+\lambda)g(V,W)(X+h'X+\lambda\eta(X)\xi).\]
Using \eqref{Rcanonic_parallel}, a straightforward computation shows that
\begin{equation}\label{Rtilde=0}
\widetilde R_{XY}V=\widetilde R_{VW}X=\widetilde R_{VX}Y=\widetilde R_{XV}W=0.
\end{equation}
We know that the distribution $[\xi]\oplus[\lambda]$ has totally geodesic leaves with constant Riemannian curvature $\kappa-2\alpha^2\lambda$, and applying \eqref{Rcanonic_parallel} again, we have
\begin{equation}\label{Rtilde=01}
\widetilde R_{UV}W=0
\end{equation}
for any $U,V,W\in[\lambda]$. It remains to analyze the curvatures $\widetilde R_{XY}Z$, with $X,Y,Z\in TB$. Considered the eigenvalue $-\lambda$, we regard $B$ as the warped product  $B'\times_{f'}M_{-\lambda}$ such that $TB'=[\xi]\oplus\bigoplus_{\mu\ne\pm\lambda}[\mu]$, $TM_{-\lambda}=[-\lambda]$ and $f'=c' e^{\alpha(1-\lambda)t}$, $c'>0$. Analogous computations give \eqref{Rtilde=0} and \eqref{Rtilde=01} for any $U,V,W\in[-\lambda]$ and $X,Y\in TB'$.
In fact this argument can be applied for each non-vanishing eigenvalue of $h'$, proving that if $0$ is a simple eigenvalue, then the curvature tensor $\widetilde R$ vanishes on $M^{2n+1}$.

If $0$ has multiplicity greater than $1$, we have to analyze the curvature tensor $\widetilde R$ on the integral submanifolds of the distribution $[\xi]\oplus[0]$. These leaves are endowed with an almost $\alpha$-Kenmotsu structure with vanishing operator $h'$ and such that the integral manifolds of $[0]$ are K\"ahler. Hence, the leaves of $[\xi]\oplus[0]$ are $\alpha$-Kenmotsu manifolds and, by Proposition \ref{Rtilde0}, the curvature tensor $\widetilde R$ vanishes on them if and only if it is parallel with respect to $\widetilde\nabla$, or equivalently the leaves have constant Riemannian curvature $k=-\alpha^2$.
\end{proof}

\begin{rem}
\emph{Differently from the normal case, which is described in Proposition \ref{Rtilde0}, in the hypotheses of Theorem \ref{theo_curv}, conditions a) and b) are not equivalent to the local Riemannian symmetry. Indeed in this case, combining \eqref{RXYxi} and \eqref{nablahsym_0} it follows that the Riemannian curvature satisfies $R_{XY}\xi=0$ for any $X,Y\in\mathcal D$. By Theorem \ref{theo_symm} it follows that $M^{2n+1}$ is locally symmetric if and only if $\mbox{\emph{Sp}}(h')=\{0,1,-1\}$, with $0$ simple eigenvalue.}
\end{rem}

\begin{cor}\label{Rtilde01}
Let $(M^{2n+1},\varphi,\xi,\eta,g)$ be an almost $\alpha$-Kenmotsu manifold such that $h'\ne0$ and $\xi$ belongs to the $(\kappa,\mu)'$-nullity distribution, $\mu=-2\alpha^2$. Then $\widetilde R=0$ and the Riemannian curvature tensor is given by
\begin{align}\label{Rkmu_alpha}
R_{XY}Z&= \kappa\,\eta(Z)(\eta(Y)X-\eta(X)Y)+\kappa (g(Y,Z)\eta(X)-g(X,Z)\eta(Y))\xi\\
&\quad+\alpha^2(g(Y-h'Y,Z)\eta(X)-g(X-h'X,Z)\eta(Y))\xi\nonumber\\
&\quad+\alpha^2\eta(Z)(\eta(Y)(X-h'X)-\eta(X)(Y-h'Y))\nonumber\\
&\quad-\alpha^2(g(Y+h'Y,Z)(X+h'X)+g(X+h'X,Z)(Y+h'Y))\nonumber
\end{align}
for any $X,Y,Z\in\frak{X}(M)$.
\end{cor}

\begin{proof}
The operator $h'$ is $\eta$-parallel and satisfies $\nabla_\xi h'=0$. The eigenvalue $0$ of $h'$ is simple and Theorem \ref{theo_curv} implies that the curvature tensor $\widetilde R$ vanishes. Moreover, since ${h'}^2=\lambda^2(I-\eta\otimes\xi)$, with $\lambda^2=-1-\frac{\kappa}{\alpha^2}$, applying \eqref{Rcanonic_parallel} we get \eqref{Rkmu_alpha}.
\end{proof}

\section{The local classification}

\begin{theo}\label{theo_equiv}
Let $(M^{2n+1},\varphi,\xi,\eta,g)$ and $(\bar M^{2n+1},\bar\varphi,\bar\xi,\bar\eta,\bar g)$ be $CR$-integrable almost $\alpha$-Kenmotsu manifolds with canonical connections $\widetilde\nabla$ and $\widetilde{\bar\nabla}$ respectively. Let us suppose that $\widetilde\nabla$ and $\widetilde{\bar\nabla}$ have parallel torsion and curvature tensors and the operators $h'$ and $\bar h'$ associated to the structures have the same eigenvalues with the same multiplicities.
Then $M^{2n+1}$ and $\bar M^{2n+1}$ are locally equivalent as almost contact metric manifolds.
\end{theo}
\begin{proof}
Let us suppose that $h'$ and $\bar h'$ have the same eigenvalues with the same multiplicities. Fixed two points $p\in M^{2n+1}$ and $q\in \bar M^{2n+1}$, we can choose orthonormal bases $\{\xi_p,e_1,\ldots,e_n,\varphi_p e_1,\ldots,\varphi_p e_n\}$ of $T_pM^{2n+1}$ and $\{\bar\xi_q,\bar e_1,\ldots,\bar e_n,\bar\varphi_q \bar e_1,\ldots,\bar\varphi_q \bar e_n\}$ of $T_q\bar M^{2n+1}$ in such a way that, for any $i=1,\ldots,n$, $e_i$ and $\bar e_i$ are eigenvectors of $h'_p$ and $\bar h'_q$, respectively, with eigenvalue $\lambda_i$, while $\varphi_p e_i$ and  $\bar \varphi_q \bar e_i$ are eigenvectors of $h'_p$ and $\bar h'_q$ with eigenvalue $-\lambda_i$.
We define a linear isometry $F:T_pM^{2n+1}\to T_q\bar M^{2n+1}$ such that
\[F(\xi_p)=\bar\xi_q,\qquad F(e_i)=\bar e_i,\qquad F(\varphi_p e_i)=\bar\varphi_q\bar e_i\]
for every $i=1,\ldots,n$. Then, we have
\[F^*\bar\eta_q=\eta_p,\qquad F^*\bar\varphi_q=\varphi_p,\qquad F^*\bar h'_q=h'_p.\]
From Theorem \ref{tildenabla}, the torsion tensors satisfy $F^*\widetilde{\bar T_q}=\widetilde T_p$. On the other hand, the curvatures $\widetilde{\bar R}$ and $\widetilde R$ vanish. It follows that there exists an affine diffeomorphism $f$ of a neighborhood $U$ of $p$ onto a neighborhood $V$ of $q$ such that $f(p)=q$ and the differential of $f$ at $p$ coincides with $F$ \cite{KN}.
The local diffeomorphism $f$ maps the structure tensors $(\varphi,\xi,\eta,g)$ to $(\bar\varphi,\bar\xi,\bar\eta,\bar g)$, since they are parallel with respect to $\widetilde\nabla$ and $\widetilde{\bar\nabla}$ respectively. Hence $M^{2n+1}$ and $\bar M^{2n+1}$ are locally equivalent as almost contact metric spaces.
\end{proof}

\begin{theo}\label{theo_class}
Let $(M^{2n+1},\varphi,\xi,\eta,g)$ and $(\bar M^{2n+1},\bar\varphi,\bar\xi,\bar\eta,\bar g)$ be $CR$-integrable almost $\alpha$ and almost $\bar\alpha$-Kenmotsu manifolds with canonical connections $\widetilde\nabla$ and $\widetilde{\bar\nabla}$ respectively. Let us suppose that $\widetilde\nabla$ and $\widetilde{\bar\nabla}$ have parallel torsion and curvature tensors. Then $M^{2n+1}$ and $\bar M^{2n+1}$ are locally equivalent as almost contact metric manifolds, up to $\mathcal D$-homothetic deformations, if and only if the operators $h'$ and $\bar h'$ associated to the structures have the same eigenvalues with the same multiplicities.
\end{theo}
\begin{proof}
Let us suppose that $h'$ has eigenvalues $0,\lambda_1,\ldots,\lambda_n,-\lambda_1,\ldots,-\lambda_n$, with $0\leq\lambda_i\leq \lambda_j$ for any $i\leq j$. Analogously, let $0,\bar\lambda_1,\ldots,\bar\lambda_n,-\bar\lambda_1,\ldots,-\bar\lambda_n$ be the eigenvalues of $\bar h'$, with $0\leq\bar\lambda_i\leq \bar\lambda_j$ for any $i\leq j$. Since a $\mathcal D$-homothetic deformation of the structure leaves the operator $h'$ invariant, if $M^{2n+1}$ and $\bar M^{2n+1}$ are locally equivalent up to $\mathcal D$-homothetic deformations, then $\lambda_i=\bar \lambda_i$ for any $i=1,\ldots, n$.
Conversely, let us suppose $\lambda_i=\bar \lambda_i$ for any $i=1,\ldots, n$. We apply a $\mathcal D$-homothetic deformation with constant $\beta= \frac{\alpha}{\bar\alpha}$ to the structure $(\varphi,\xi,\eta,g)$, thus obtaining an almost $\bar\alpha$-Kenmotsu structure $(\varphi_1,\xi_1,\eta_1, g_1)$ on $M^{2n+1}$ for which the canonical connection has parallel torsion and vanishing curvature and the operator $h_1'$ has eigenvalues $0,\lambda_1,\ldots,\lambda_n,-\lambda_1,\ldots,-\lambda_n$. From Theorem \ref{theo_equiv} it follows that $(M^{2n+1},\varphi_1,\xi_1,\eta_1, g_1)$ and $(\bar M^{2n+1},\bar \varphi,\bar \xi,\bar \eta,\bar g)$ are locally equivalent as almost contact metric spaces.
\end{proof}

For any odd dimension $2n+1$ and for any nonnegative and not all vanishing real numbers $\lambda_1,\ldots,\lambda_n$, we give an example of a $CR$-integrable almost $\alpha$-Kenmotsu manifold whose canonical connection has parallel torsion and vanishing curvature and such that the operator $h'$ has eigenvalues $0,\lambda_1,\ldots,\lambda_n,-\lambda_1,\ldots,-\lambda_n$. The example is given by a Lie group endowed with a left invariant almost $\alpha$-Kenmotsu structure.

Let $G$ be the connected and simply connected Lie group of real matrices of the form
\[A=\begin{pmatrix}
e^{-\alpha(1+\lambda_1)t}&0&\cdots&0&0&0&x_1\\
0&e^{-\alpha(1-\lambda_1)t}&\cdots&0&0&0&y_1\\
\vdots&\vdots&\ddots&\vdots&\vdots&\vdots&\vdots&\\
0&0&\cdots&e^{-\alpha(1+\lambda_n)t}&0&0&x_n\\
0&0&\cdots&0&e^{-\alpha(1-\lambda_n)t}&0&y_n\\
0&0&\cdots&0&0&1&t\\
0&0&\cdots&0&0&0&1\end{pmatrix},\]
which is a subgroup of the affine group $\mbox{\emph{Aff}}(2n+1,\mathbb{R})$ and
whose Lie algebra $\frak{g}$ is given by the real matrices
\[X=\begin{pmatrix}
-\alpha(1+\lambda_1)c&0&\cdots&0&0&0&a_1\\
0&-\alpha(1-\lambda_1)c&\cdots&0&0&0&b_1\\
\vdots&\vdots&\ddots&\vdots&\vdots&\vdots&\vdots&\\
0&0&\cdots&-\alpha(1+\lambda_n)c&0&0&a_n\\
0&0&\cdots&0&-\alpha(1-\lambda_n)c&0&b_n\\
0&0&\cdots&0&0&0&c\\
0&0&\cdots&0&0&0&0\end{pmatrix}.\]
For any $i=1,\ldots,n$ denote by $X_i$ and $Y_i$ the matrices in $\frak{g}$ whose coefficients are all vanishing except for $a_i=1$ and $b_i=1$ respectively. Let $\xi$ be the matrix corresponding to $c=1$ and $a_i=b_i=0$, $i=1,\dots,n$. Then $\{\xi,X_1,\ldots,X_n,Y_1,\ldots,Y_n\}$ is a basis of $\frak{g}$ which satisfies
\[[\xi,X_i]=-\alpha(1+\lambda_i)X_i,\qquad [\xi,Y_i]=-\alpha(1-\lambda_i)Y_i,\]\vspace{-0.55cm}
\[[X_i,X_j]=[X_i,Y_j]=[Y_i,X_j]=[Y_i,Y_j]=0.\]
The above relations imply that $G$ is a solvable non-nilpotent Lie group.
We consider the
endomorphism $\varphi:\mathfrak{g}\rightarrow \mathfrak{g}$ and the
1-form $\eta: \mathfrak{g}\rightarrow\mathbb{R}$ such that
\begin{eqnarray*}
\varphi(\xi)=0,\quad
\varphi(X_i)=Y_i,\quad\varphi(Y_i)=-X_i,\quad
\eta(\xi)=1,\quad\eta(X_i)=\eta(Y_i)=0,
\end{eqnarray*}
for any $i=1,\ldots,n$, and denote by $g$ the inner product on
$\mathfrak{g}$ such that the basis $\{\xi,X_i,Y_i\}$ is
orthonormal.
The tensors defined on $\frak g$ determine a left invariant $CR$-integrable almost $\alpha$-Kenmotsu structure $(\varphi,\xi,\eta,g)$ on $G$. Each $X_i$ is an eigenvector of $h'$ with eigenvalue $\lambda_i$, while each $Y_i$ is eigenvector with eigenvalue $-\lambda_i$. Moreover, the canonical connection $\widetilde\nabla$ coincides with the left invariant connection on the Lie group, which has vanishing curvature and parallel torsion. Indeed, denoting by $\nabla'$ the left invariant connection on $G$, the structure tensor fields $\varphi$, $\xi$, $\eta$, $g$ are parallel with respect to $\nabla'$. Since the torsion $T'$ is given by $2T'(X,Y)=-[X,Y]$ for any $X,Y\in\frak{g}$,  $T'$ satisfies a), b), c) of Theorem \ref{tildenabla}, so that $\nabla'$ coincides with the canonical connection associated to the structure $(\varphi,\xi,\eta,g)$.

Finally, considering the coordinate system $\{t,x_1,y_1,\ldots,x_n,y_n\}$ on the Lie group $G$, we have
\[X_i=e^{-\alpha(1+\lambda_i)t}\frac{\partial}{\partial x_i},\qquad Y_i=e^{-\alpha(1-\lambda_i)t}\frac{\partial}{\partial y_i},\]
\begin{equation}\label{local1}
\xi=\frac{\partial}{\partial t},\qquad \eta=dt,\qquad \varphi=\sum_{i=1}^ne^{2\alpha\lambda_i t}dx_i\otimes\frac{\partial}{\partial y_i}-\sum_{i=1}^ne^{-2\alpha\lambda_i t}dy_i\otimes\frac{\partial}{\partial x_i},
\end{equation}
\begin{equation}\label{local2}
g=dt\otimes dt+\sum_{i=1}^ne^{2\alpha(1+\lambda_i)t}dx_i\otimes dx_i+\sum_{i=1}^ne^{2\alpha(1-\lambda_i)t}dy_i\otimes dy_i.
\end{equation}
Therefore, from Theorem \ref{theo_equiv} we have the following
\begin{prop}\label{local_prop}
Let $(M^{2n+1},\varphi,\xi,\eta,g)$ be an almost $\alpha$-Kenmotsu manifold. Then $M^{2n+1}$ is a $CR$-integrable almost $\alpha$-Kenmotsu manifold with canonical connection $\widetilde\nabla$ satisfying $\widetilde\nabla\widetilde T=0$  and $\widetilde\nabla\widetilde R=0$ if and only if for any point $p\in M^{2n+1}$ there exists an open neighbourhood with local coordinates $\{t,x_1,y_1,\ldots,x_n,y_n\}$ on which \eqref{local1} and \eqref{local2} hold.
\end{prop}

As a consequence of Theorem \ref{theo_class}, we can associate to each almost $\alpha$-Kenmotsu $(\kappa,\mu)'$-space $M^{2n+1}$, with $h'\ne0$, the real number
\[I_{M^{2n+1}}=\frac{\kappa}{\alpha^2},\]
which classifies such spaces up to $\mathcal D$-homothetic deformations, as stated in the following Theorem.

\begin{theo}
Let $(M^{2n+1},\varphi,\xi,\eta,g)$ be an almost $\alpha$-Kenmotsu manifold with $h'\ne0$ and $\xi$ belonging to the $(\kappa,\mu)'$-nullity distribution, $\mu=2\alpha^2$, and let $(\bar M^{2n+1},\bar \varphi,\bar \xi,\bar \eta,\bar g)$ be an almost $\bar \alpha$-Kenmotsu manifold with $\bar h'\ne0$ and $\bar\xi$ belonging to the $(\bar\kappa,\bar\mu)'$-nullity distribution, $\bar\mu=-2\bar\alpha^2$. Then, $M^{2n+1}$ and $\bar M^{2n+1}$ are locally equivalent up to a $\mathcal D$-homothetic deformation, as almost contact metric spaces, if and only if $I_{M^{2n+1}}=I_{\bar M^{2n+1}}$.
\end{theo}
\begin{proof}
The operators $h'$ and $\bar h'$ have eigenvalues $0,\lambda,-\lambda$ and $0,\bar\lambda,-\bar\lambda$ respectively, where $0$ is simple, $\lambda=\sqrt{-1-\frac{\kappa}{\alpha^2}}$ and $\bar\lambda=\sqrt{-1-\frac{\bar\kappa}{\bar\alpha^2}}$. The result immediately follows from Theorem \ref{theo_class}. In particular, if $I_{M^{2n+1}}=I_{\bar M^{2n+1}}$, we have to apply a $\mathcal D$-homothetic deformation to the structure $(\varphi,\xi,\eta,g)$ with constant $\beta=\frac{\alpha}{\bar\alpha}= \sqrt{\frac{\kappa}{\bar\kappa}}$.
\end{proof}

Notice that $I_{M^{2n+1}}<-1$ since $\kappa<-\alpha^2$. By Corollary \ref{cor_symm}, $M^{2n+1}$ is locally symmetric if and only if $I_{M^{2n+1}}=-2$. For any dimension $2n+1$ and for any value of the invariant $I<-1$, an explicit example of these manifolds is given by the Lie group $G$ described above, with $\lambda_1=\ldots=\lambda_n=\lambda$. Indeed, considered the left invariant almost $\alpha$-Kenmotsu structure $(\varphi,\xi,\eta,g)$, by Proposition \ref{prop_0lambda}, the characteristic vector field $\xi$ belongs $(\kappa,\mu)'$-nullity distribution, with $\kappa= -\alpha^2(1+\lambda^2)$ and
$\mu=-2\alpha^2$. The invariant is
\[I_G=-1-\lambda^2\]
which attains any real value smaller than $-1$.

\begin{rem}
\emph{In \cite{Bo_classification} E. Boeckx introduces a scalar invariant which classifies, up to $\mathcal D$-homothetic deformations, non-Sasakian contact metric manifolds whose characteristic vector field belongs to a $(\kappa,\mu)$-nullity distribution. The proof of the equivalence theorem involves the Levi-Civita connection and the properties of the covariant derivatives of the Riemannian curvature $R$ and the structure tensors $\varphi,\xi,\eta,g$. Here, the classification theorem for almost $\alpha$-Kenmotsu $(\kappa,\mu)'$-spaces, is obtained as a consequence of the more general Theorem \ref{theo_class}, which involves the canonical connection and the parallelism  with respect to it of the torsion $\widetilde T$, the curvature $\widetilde R$ and the structure tensors. One could wonder if it is possible to prove the equivalence theorem for non-Sasakian $(\kappa,\mu)$-contact metric spaces by using the Tanaka-Webster connection. The answer is negative, since in this case the torsion tensor $\widetilde T$ and the curvature tensor $\widetilde R$ of the Tanaka-Webster connection $\widetilde\nabla$ in general are not parallel with respect to $\widetilde \nabla$; this happens if and only if $\mu=2$, as proved in \cite{BoCho_symm}.}
\end{rem}

We conclude the analysis of the local geometry of almost $\alpha$-Kenmotsu $(\kappa,\mu)'$-spaces with the following result.
\begin{prop}
Let $(M^{2n+1},\varphi,\xi,\eta,g)$ be an almost $\alpha$-Kenmotsu manifold such that $h'\ne0$ and $\xi$ belongs to the $(\kappa,\mu)'$-nullity distribution, $\mu=-2\alpha^2$. Let $g'$ be the Riemannian metric locally defined by the $\mathcal D$-conformal change
\[g'=e^{-2\alpha t}g+(1-e^{-2\alpha t})\eta\otimes \eta.\]
Then $(\varphi,\xi,\eta,g')$ is an almost cosymplectic structure such that $\xi$ belongs to the $\kappa_c$-nullity distribution, with $\kappa_c=\kappa+\alpha^2$.
\end{prop}
\begin{proof}
The fundamental $2$-form $\Phi'$ associated to the structure $(\varphi,\xi,\eta,g')$ is locally given by $\Phi'=e^{-2\alpha t}\Phi$ and thus $d\Phi'=0$, so that $(\varphi,\xi,\eta,g')$ is an almost cosymplectic structure. By Proposition \ref{local_prop}, for any point $p\in M^{2n+1}$ there exist local coordinates $\{t,x_1,y_1,\ldots,x_n,y_n\}$, such that\[\xi=\frac{\partial}{\partial t},\qquad \eta=dt,\qquad \varphi=e^{2\alpha\lambda t}\sum_{i=1}^ndx_i\otimes\frac{\partial}{\partial y_i}-e^{-2\alpha\lambda t}\sum_{i=1}^ndy_i\otimes\frac{\partial}{\partial x_i},\]
\[g=dt\otimes dt+e^{2\alpha(1+\lambda)t}\sum_{i=1}^ndx_i\otimes dx_i+e^{2\alpha(1-\lambda)t}\sum_{i=1}^ndy_i\otimes dy_i,\]
with $\lambda=\sqrt{-1-\frac{\kappa}{\alpha^2}}$. Hence, the Riemannian metric $g'$ is locally given by
\[g'=dt\otimes dt+e^{2\alpha\lambda t}\sum_{i=1}^ndx_i\otimes dx_i+e^{-2\alpha\lambda t}\sum_{i=1}^ndy_i\otimes dy_i.\]
By a result of P. Dacko \cite{Da}, it follows that $(\varphi,\xi,\eta,g')$ is an almost cosymplectic structure such that $\xi$ belongs to the $\kappa_c$-nullity distribution, with $\kappa_c=-\lambda^2\alpha^2=\kappa+\alpha^2<0$.
\end{proof}

\noindent Dipartimento di Matematica\\
Universit\`a degli Studi di Bari Aldo Moro\\
Via E. Orabona 4, 70125 Bari, Italy\\
dileo@dm.uniba.it

\begin{thebibliography}{99}
\bibitem{besse} \textsc{A.L. Besse}, Einstein manifolds,
Springer-Verlag, Berlin, 1987.
\bibitem{Bl1} \textsc{D.E. Blair}, Contact manifolds in Riemannian geometry,
Lecture Notes in Math. 509, Springer-Verlag, Berlin - New York, 1976.
\bibitem{Bl2} \textsc{D.E. Blair}, Riemannian geometry of contact and
symplectic manifolds, Birkh\"auser, Boston, 2002.
\bibitem{Bo_classification} \textsc{E. Boeckx}, A full classification of contact metric $(\kappa,\mu)$-spaces, Illinois J. Math. 44 (2000), no. 1, 212--219.
\bibitem{BoCho_symm} \textsc{E. Boeckx and J.T. Cho}, Pseudo-Hermitian symmetries, Israel J. Math. 166 (2008), 125--145.
\bibitem{Da} \textsc{P. Dacko}, On almost cosymplectic manifolds with the structure vector field $\xi$ belonging to the $k$-nullity distribution, Balkan J. Geom. Appl. 5 (2000), no. 2, 47--60.
\bibitem{DP} \textsc{G. Dileo and A.M. Pastore}, Almost Kenmotsu manifolds and local symmetry, Bull. Belg. Math. Soc. Simon Stevin 14 (2007), 343--354.
\bibitem{DPnull} \textsc{G. Dileo and A.M. Pastore}, Almost Kenmotsu manifolds and nullity distributions, J. Geom. 93 (2009), 46--61.
\bibitem{DPeta} \textsc{G. Dileo and A.M. Pastore}, Almost Kenmotsu manifolds with a condition of $\eta$-parallelism, Differential Geom. Appl. 27 (2009), 671--679.
\bibitem{Hi} \textsc{S. Hiepko}, Eine innere Kennzeichnung der verzerrten Produkte,
Math. Ann. 241 (1979), 209--215.
\bibitem{Ia} \textsc{S. Ianus}, Sulle variet\g a  di Cauchy-Riemann, Rend. Accad. Sci. Fis. Mat. Napoli (4) 39 (1972), 191--195.
\bibitem{JanVan} \textsc{D. Janssens and L. Vanhecke}, Almost contact structures and curvature tensors, Kodai Math J. 4 (1981), 1--27.
\bibitem{Ken} \textsc{K. Kenmotsu}, A class of almost contact Riemannian manifolds,
T\^ohoku Math. J. 24 (1972), 93--103.
\bibitem{KP} \textsc{T.W. Kim and H.K. Pak}, Canonical foliations of certain classes
of almost contact metric structures, Acta Math. Sin. (Engl. Ser.)
21, No. 4 (2005), 841--846.
\bibitem{KN} \textsc{S. Kobayashi and K. Nomizu}, Foundations of differential geometry I, Interscience Publishers, New York - London, 1963.
\bibitem{OLS1} \textsc{Z. Olszak}, Locally conformal almost cosymplectic manifolds,
Colloq. Math. 57 (1989), 73--87.
\bibitem{ONeill} \textsc{B. O'Neill}, Semi-Riemannian geometry, Academic
Press, New York, 1983.
\bibitem{Tan1} \textsc{S. Tanno}, The topology of contact Riemannian manifolds, Illinois J. Math. 12 (1968), 700--717.
\bibitem{Ta} \textsc{S. Tanno}, The automorphism groups of almost contact Riemannian manifolds, T\^ohoku Math. J. 21 (1969), 21--38.\\
\end{thebibliography}
\end{document}